\documentclass[a4paper,graybox]{svmult}

\usepackage{mathptmx}       
\usepackage{helvet}         
\usepackage{courier}        
\usepackage{type1cm}        
\usepackage{makeidx}         
\usepackage{graphicx}        
\usepackage{multicol}        
\usepackage[bottom]{footmisc}
\usepackage{url}

\makeindex             

\usepackage{amsmath,latexsym,amssymb}

\DeclareMathOperator{\Co}{co}
\DeclareMathOperator{\Expectation}{E} 
\DeclareMathOperator{\Ideal}{Ideal}
\DeclareMathOperator{\Image}{im} 
\DeclareMathOperator{\Prob}{P}
\DeclareMathOperator{\Span}{Span}
\DeclareMathOperator{\Rank}{Rank}
\DeclareMathOperator{\Supp}{Supp}
\newcommand{\convexof}[1]{\Co\left(#1\right)}
\newcommand{\euler}{\mathrm{e}}
\newcommand{\expectat}[2]{{\Expectation}_{#1}\left[#2\right]}
\newcommand{\expof}[1]{\exp\left(#1\right)}
\newcommand{\idealof}[1]{\Ideal\left(#1\right)}
\newcommand{\integers}{\mathbb{Z}}
\newcommand{\lattice}[1]{\mathcal L^\star \left(#1\right)}
\newcommand{\probof}[1]{\Prob\left(#1\right)}
\newcommand{\rationals}{\mathbb{Q}}
\newcommand{\reals}{\mathbb{R}}
\newcommand{\rankof}[1]{\Rank\left(#1\right)}

\newcommand{\setof}[2]{\left\{#1 : #2 \right\}}
\newcommand{\set}[1]{\left\{#1\right\}}
\newcommand{\spanof}[1]{\Span\left(#1\right)}
\newcommand{\sspace}[1]{\mathcal #1}
\newcommand{\supp}[1]{\Supp{#1}}

\begin{document}

\title*{A note on the border of an exponential family 
}
\titlerunning{The border of an exponential family}
\author{Luigi Malag\`o and Giovanni Pistone}
\authorrunning{L. Malag\`o, G. Pistone}
\institute{Luigi Malag\`o \at DEI Politecnico di Milano, Italy. \email{malago@elet.polimi.it}
\and Giovanni Pistone \at Collegio Carlo Alberto, Moncalieri, Italy. \email{giovanni.pistone@gmail.com}}
\maketitle

\abstract{Limits of densities belonging to an exponential family appear in many applications, {e.g.} Gibbs models in Statistical Physics, relaxed combinatorial optimization, coding theory, critical likelihood computations, Bayes priors with singular support, random generation of factorial designs. We discuss the problem from the methodological point of view in the case of a finite state space. We prove two characterizations of the limit distributions, both based on a suitable description of the marginal polytope (convex hull of canonical statistics' values). First, the set of limit densities is equal to the set of conditional densities given a face of the marginal polytope. Second, in the lattice case there exists a parametric presentation, in monomial form, of the closure of the statistical model.}

\keywords{Algebraic Statistics, Convex Support, Extended Exponential Family, Statistical Modeling.}

\section{Background}
\label{sec:BG}
We consider the \emph{exponential family} defined by the family of densities 
 \begin{equation}
  \label{eq:1}
  p(x;\theta) = \expof{\sum_{j=1}^m \theta_j T_j(x) - \psi(\theta)}, \quad \theta \in \reals^m,
\end{equation}
on a finite state space $(\sspace X,\mu)$ with $n = \# \sspace X$ points and reference measure $\mu$. Many monographs have been devoted to the study of this important class of statistical models, {e.g.} \cite{barndorff-nielsen:78,brown:86,letac:92}. In this section we have collected facts from this theory and its algebraic version in order to
introduce to our result discussed in Section \ref{sec:hilbert}.

Different exponential families could represent the same statistical model. Consider the orthogonal decomposition $\spanof{1,T_1,\dots,T_m} = 1\oplus V \subset L^2(\sspace X,\mu)$. In fact, $V \subset L_0^2(\sspace X,\mu) $. For each density $p$ in the exponential family \eqref{eq:1} there exists a unique $v \in V$ such that $p(x) = \euler^{v - K_\mu(v)}$, see \cite{pistone|sempi:95,cena|pistone:2007}. 

The \emph{canonical statistics}
\begin{equation*}
  T = (T_1,\dots,T_m) \colon \sspace X \to \sspace Y = T(\sspace X) \subset \reals^m
\end{equation*}
map the statistical model \eqref{eq:1} to the \emph{canonical exponential family} 
\begin{equation}
  \label{eq:1-S}
  p(y;\theta) = \expof{\sum_{j=1}^m \theta_j y_j - \psi(\theta)}, \quad \theta \in \reals^m,
\end{equation}
where the new state space is $(\sspace Y,\nu)$, with $\nu = \mu \circ T^{-1}$. In Equation \eqref{eq:1-S}, the canonical statistics are coordinate projections $y \mapsto y_j$, $j=1,\dots,m$.

\subsection{Monomial and implicit presentations}

Other useful parameterization of the exponential family \eqref{eq:1} are available, in particular the mean parameterization which shall be discussed in Section \ref{sec:marginal}. In this paper we focus on a less known parameterization, i.e. the \emph{monomial parameterization}, which is obtained from \eqref{eq:1} by introducing the exponentials $t_j = \euler^{\theta_j}$ of each canonical parameter $\theta_j$, $j=1,\dots,m$,
 \begin{equation}
  \label{eq:1-T}
  p(x;t) \propto \prod_{j=1}^m t_j^{T_j(x)}, \quad t \in \reals_>^m.
\end{equation}
This presentation is especially useful in the lattice case, {i.e.} when the canonical statistics are integer valued. This is the case which has been studied with the methods of Algebraic Statistics, see {e.g.} \cite[Sec.  6.9]{pistone|riccomagno|wynn:2001}, \cite{geiger|meek|sturmfels:2006}.

While Equations \eqref{eq:1} and \eqref{eq:1-T} are equivalent for positive densities, an interesting phenomenon appears if the conditions $t_j >0$ are relaxed to $t_j \ge 0$. In such a case, \eqref{eq:1-T} makes sense and an extension of the original model is obtained, see \cite{rapallo:2003thesis,rapallo:2007toric}. For example, assume we let just one of the $t_j$'s, say $t_1$, to be zero. It follows that the corresponding unnormalized density is zero if $T_1(x) \ne 0$ and is positive for $T_1(x) = 0$, giving rise to densities with support $\set{T_1 = 0}$ which form a new exponential family. Thus, the exponential family \eqref{eq:1} is extended to include exponential families with defective support. Unfortunately, such extension depends on the canonical statistics used to describe the statistical model as an exponential family. For example, if the chosen canonical statistics are never zero, no such extension is possible. 

Statistical models of type \eqref{eq:1} admit an \emph{implicit representation}, see \cite{pistone|riccomagno|wynn:2001,pistone:2009SL}. Let $1 \oplus V = \spanof{T_0=1, T_1, \dots, T_m}$ be the linear space generated by the canonical statistics together with the constant $1$, and let $w_1, \dots, w_l$ be a linear basis of the orthogonal space $(1\oplus V)^\perp$, {i.e.}, $1, T_1,\dots,T_m,w_1,\dots,w_l$ is a linear basis of  $L^2(\sspace X,\mu)$ and
\begin{equation*}
  \sum_{x \in \sspace X} w_i(x) T_j(x) \mu(x) = 0, \quad i=1,\dots,l,\quad j=0,\dots,m.
\end{equation*}
If we introduce the $(m+1)\times n$ matrix $A = [T_j(x)\mu(x)]$, $j=0,\dots,m$, $x \in \sspace X$, $T_0 = 1$, then $\spanof{w_1, \dots, w_l} = \ker A$. The case where $A$ is integer valued is discussed in \cite{geiger|meek|sturmfels:2006}. The general case is discussed in \cite{rauh|kahle|ay:09}.

Since $\log p(\cdot;\theta)$ is an affine function of the canonical statistics $T_j$'s, a density $p$ belongs to the exponential model \eqref{eq:1} if and only if $p$ is a positive density of $(\sspace X,\mu)$ and
\begin{equation}
  \label{eq:1-D}
\sum_{x\in\sspace X} w(x) \mu(x) \log p(x) = 0, \quad w \in \spanof{w_1,\dots,w_l}.
\end{equation}
More precisely, if $p=p(\cdot;\theta)$ in \eqref{eq:1} for a $\theta$, then \eqref{eq:1-D} holds true; vice versa, if $\sum_{x\in\sspace X} w(x) \mu(x) \log p(x) = 0$ holds true for $w = w_i$, $i=1,\dots,l$, then $p=p(\cdot;\theta)$ for some $\theta$.

Equation \eqref{eq:1-D} is equivalent to the following equation
\begin{equation}
  \label{eq:1-M}
\prod_{x\in\sspace X} p(x)^{w(x)} = 1, \quad {w}/{\mu} \in \spanof{w_1,\dots,w_l},
\end{equation}
or, clearing the denominators,
\begin{equation}
  \label{eq:1-MM}
\prod_{x\colon w(x) > 0} p(x)^{w^+(x)} = \prod_{x\colon w(x) < 0} p(x)^{w^-(x)}, \quad w/{\mu} \in \spanof{w_1,\dots,w_l},
\end{equation}
where $w=w^+-w^-$ and $w^+, w^- \ge 0$.
Equation \eqref{eq:1-MM} makes sense outside the exponential model, {i.e.} if we assume $p(x) \ge 0$. Assume $\sspace X_0 = \supp p$ is strictly contained in $\sspace X$ and satisfies Equation \eqref{eq:1-M}. Therefore, $p$ belongs to the exponential model associated to the space $V_0$, with $1\oplus V_0 = \spanof{w_{1|\sspace X_0},\dots,w_{l|\sspace X_0}}^\perp \subset L^2\left(\sspace X_0, \mu_{|\sspace X_0}\right)$.  
\subsection{Toric statistical models}
From now on we assume that the $m\times\sspace X$ matrix $A = [T_j(x)]_{j = 1,\dots,d; x\in\sspace X}$, is nonnegative integer valued. The nonnegativity assumption does not restrict the class of model we consider. We define
\begin{equation*}
  \lattice A = \setof{y \in \integers^{\sspace X}}{y\ne 0, Ay = 0}
\end{equation*}
be the \emph{lattice} of $A$. We denote by $A(x)$, $x \in \sspace X$, the columns of $A$. The model \eqref{eq:1-T} is written $p(x;t)=t^{A(x)}=t_1^{A_1(x)}\cdots t^{A_m(x)}$ and it is called \emph{$A$-model}.

Consider  the homomorphism $\tau$ from the polynomial ring $\rationals[q(x) \colon x \in \sspace X]$ into the polynomial ring $\rationals[t_j \colon j = 1,\dots,d]$ defined by
\begin{equation*}
\tau \colon q(x) \mapsto \prod_{j=1}^d t_j^{T_j(x)} = t^{A(x)}, \quad x \in \sspace X.
\end{equation*}

The kernel of $\tau$ is a polynomial ideal $\idealof A$, which is called the \emph{toric ideal} of $A$. It is proved in \cite{sturmfels:1996} that $\idealof A$ is generated as a vector space by the binomials
\begin{equation*}
  \prod_{x\in\sspace X} q(x)^{w_+(x)} - \prod_{x\in\sspace X} q(x)^{w_-(x)}, \quad k \in \lattice A
\end{equation*}
and it is generated as an ideal by a finite subset of such binomials, i.e. the binomials where $k$ is an element of the Graver basis of $\lattice A$. Note that the binomials are homogeneous if, and only if, $1 \in \spanof A$.

Assume now that $t_1,\dots,t_d$ take nonnegative and not all zero real values and consider the parameterization
\begin{equation*}
q(x) = t^{A(x)}, \quad x \in \sspace X, \quad t \in Q = \reals_+^d \setminus \set{0}.  
\end{equation*}

Note that $t^{A(x)} = \prod_{j \colon A_j(x) \ne 0} t_j^{A_j(x)}$. Each $q(x)$ is nonnegative and strily positive if $t_1,\dots,t_d > 0$. Let $I$ be a subset of indices, $I \subset \set{1,\dots,d}$ such that $t_j=0$ for all $j \in I$. Then $q(x) \ne 0$ for all $x \in \sspace X$ such that $A_j(x) = 0$, $j\in I$. 

There exists at least one $x \in \sspace X$ where $q(x) \ne 0$ if, and only if, each column of $A$ contains at least one zero. In such a case, we have defined a parameterization of unnormalized probabilities $q$ with parameters in the vertex-less quadrant:
\begin{equation*}
  Q \ni t \mapsto p(x;t) =  \frac{t^{A(x)}}{\sum_{x\in\sspace X} t^{A(x)}} 
\end{equation*}

Let us study the confounding induced by such a parameterization on
strictly positive parameters. If
\begin{equation*}
  \frac{s^{A(x)}}{\sum_{x\in\sspace X} s^{A(x)}} = \frac{t^{A(x)}}{\sum_{x\in\sspace X} t^{A(x)}}, \quad x \in \sspace X,
\end{equation*}
then the unnormalized probabilities are proportional and 
\begin{equation*}
  \prod_{j = 1}^d \left(\frac{s_j}{t_j}\right)^{A_j(x)} = \text{constant}, x \in \sspace X.
\end{equation*}
or
\begin{equation*}
  \sum_{j=1}^d \left(\log s_j - \log t_j \right) A_j(x) = \text{constant}.
\end{equation*}
If, and only if, $1 \in \rankof A$, there exists vectors $c = (c_1,\dots,c_d)$ such that $A(x)c = \text{constant}$ and $\log s_j - \log t_j = c_j$, $j = 1,\dots,d$ or $s_j = \euler^{c_j} t_j$. Confounding is reduced to the confounding of uniform probability.

\subsection{Trace, closure, marginal polytope}\label{sec:marginal}

In the present section we discuss two general methods under which the reduction of the support appears, namely the trace operation and the limit operation. For each event $\mathcal S \subset \sspace X$, the \emph{trace} on $\mathcal S$ of the exponential family in \eqref{eq:1} is the exponential family defined on $(\mathcal S,\mu_{|S})$ by conditioning on $\mathcal S$.

We denote by $\mathcal M_>$ the convex set of strictly positive densities and by $\mathcal M_\ge$ the convex set of densities. Both sets are endowed with the weak topology, {i.e.}, if $p_n$, $n=1,2,\dots$, and $p$ are densities, then $\lim_{n \to \infty}p_n = p$ means $\lim_{n\to\infty} p_n(x) = p(x)$ for all $x \in \sspace X$. In general, the exponential model \eqref{eq:1} is not closed in the weak topology. The \emph{extended exponential family} is the closure in the weak topology of an exponential family \eqref{eq:1}. An extended exponential family according to this definition is a set of densities. A proper parameterization of the extended family requires the use of the expectation parameters and the identification of their range.

\begin{definition}
The \emph{convex support}, cf. e.g \cite{cenkov:1982,brown:86,csiszar|matus:2005}, or \emph{marginal polytope}, see \cite{wainwright|jordan:2008}, and also \cite{kahle:10}, of the exponential family \eqref{eq:1} is the convex hull of $\sspace Y = T(\sspace X)$, 
  \begin{equation*}
    \convexof{\Image T} = \setof{\eta \in \reals^m, \eta = \sum_{j=1}^m \lambda_j t_j}{\lambda_j \ge 0, \sum_{j=1}^m \lambda_j = 1}.
  \end{equation*}
\end{definition}

The previous set-up covers the behavior of the exponential family and its parameterization with the expectation parameters in the interior of the marginal polytope, see \cite{brown:86}. The discussion of the parameterization of the extended family requires the notion of exposed subset.

\begin{definition} 
\label{def:2}
\begin{enumerate}
\item A \emph{face} of the marginal polytope $M$ is a subset $F \subset M$ such that there exists an affine mapping $A\colon \reals^m \ni t \mapsto A(t) \in \reals$ which is zero on $F$ and strictly positive on $M\setminus F$.   
\item
A subset $S \subset \sspace X$ is \emph{exposed for the exponential family} \eqref{eq:1-S} if $S = T^{-1}(F)$ and $F$ is a face of the marginal polytope.
\end{enumerate}
\end{definition}

The following theorem is a minor improvement of known results.

\begin{theorem}\label{th:supports}
Let $\theta_n$, $n=1,2,\dots$, be a sequence of parameters in Equation \eqref{eq:1} such that for some $q \in \mathcal M_\ge$ we have $\lim_{n \to \infty} p(x;\theta_n) = q(x)$, {i.e.}, $q$ belongs to the extended exponential model.
\begin{enumerate}
\item\label{th:supports1} If the support of $q$ is full, $\set{q > 0} = \mathcal X$, then $q$ belongs to the exponential family \eqref{eq:1} for some parameter value $\theta=\lim_{n\to\infty} \theta_n$.  
\item \label{th:supports2} If the support of $q$ is defective, then the sequence $\theta_n$ is not convergent, $\supp q = \set{q > 0}$ is an exposed subset of $\sspace X$, and $q$ belongs to the trace of the exponential family on the reduced support.
\end{enumerate}
\end{theorem}
\begin{proof}
Let $\sspace X_0 = \setof{x \in \sspace X}{q(x) > 0}$, $\sspace X_1 = \setof{x \in \sspace X}{q(x)=0}$. For each $x \in \sspace X_0$, we have $\lim_{n \to \infty} \log p(x;\theta_n) = \log q(x)$ by continuity; for each $x \in \sspace X_1$, we have $\lim_{n \to \infty} \log p(x;\theta_n) = -\infty$. From \eqref{eq:1-D} we get
\begin{equation}\label{eq:argument}
  \sum_{x \in \sspace X_0} \log q(x) k(x) \mu(x) + \lim_{n \to \infty} \sum_{x \in \sspace X_1} \log p(x;\theta_n) k(x) \mu(x) = 0, 
\end{equation}
with $k \in \spanof{w_1,\dots,w_l}$. 
\begin{enumerate}
\item If the set $\sspace X_1$ is empty, then $q$ belongs to the exponential model because Equation \eqref{eq:argument} reduces to \eqref{eq:1-D}. The convergence $\lim_{n\to\infty} \eta_n = \lim_{n\to\infty} \expectat {\theta_n} T = \expectat q T = \eta$ in $M^\circ$ implies the convergence of the $\theta$ parameters ($\mod$ the identifiability constraints).
\item If the set $\sspace X_1$ is not empty, the second term of the LHS of~\eqref{eq:argument} has to be finite, so that no linear combination of the $w_i$'s can be definite in sign. Otherwise, the limit would diverge. In other words, the problem
\begin{equation} \label{eq:3}
 k : \sspace X_1\ni x \mapsto \sum_{i=1}^l \lambda_i w_i(x) \ge 0 \text{ and $k \neq 0$ for at least one $x$}
\end{equation}
is not satisfiable. By the Theorem of the alternative, see {e.g.} \cite[Ch. 15]{roman:2005}, the non satisfiability of \eqref{eq:3} is equivalent to the existence of a positive solution $u^{(1)}(x) > 0$, $x \in \sspace X_1$, to the problem
\begin{equation*}
  \label{eq:4}
  \sum_{x \in \sspace X_1} u^{(1)}(x) k(x) \mu(x) = 0, \quad k \in \spanof{w_1,\dots,w_l}.
\end{equation*}
The random variable
\begin{equation*}
  u(x) =
  \begin{cases}
    0 & \text{if $x \in \sspace X_0$}, \\
    u^{(1)}(x) & \text{if $x \in \sspace X_1$},
  \end{cases}
\end{equation*}
is orthogonal to all $w_i$'s, so that there exist $a_0, a_1, \dots, a_m$ such that
\begin{equation}
  \label{eq:5}
  u(x) = a_0 + \sum_{j=1}^m a_j T_j(x).
\end{equation}
The conclusion on the support now follows from \eqref{eq:5}. In fact, for each $t \in \sspace Y$ such that $T^{-1}(t) \in \sspace X_1$ the linear function $a_0+\sum_j a_j t_j$ is positive, while for each $t$ such that $T^{-1}(t) \in \sspace X_0$ takes value zero, so that the points in $\sspace X_1$ are the points of an exposed set of the face of $M$ identified by \eqref{eq:5}.

Finally, on the support of $q$, $\log q$ is a linear combination of the $T_j$'s being a limit in the linear space generated by those functions.\qed
\end{enumerate}
\end{proof}

\begin{theorem}\label{th:malago}
If $q$ belongs to the trace of the exponential family \eqref{eq:1} with respect to an exposed subset $S$, then $q$ belongs to the extended exponential model.  
\end{theorem}
\begin{proof} We generate sequences that admit as limit a generic density in the trace model by considering a one-dimensional (Gibbs) sub-model. Let $F$ be the face of the marginal polytope such that $S=T^{-1}(F)$ and let $A$ be an affine function such that $A(\eta)=0$ for $\eta\in F$ and $A(\eta)>0$ for $\eta\in M\setminus F$. We can chose $A$ such that $A\circ T$ belongs to the space generated by $1,T_1,\dots,T_m$, {i.e.} $A\circ T = \alpha_0 + \sum_{j=1}^m \alpha_j T_j$. We can take $\alpha_0=0$ if $1 \in \spanof{T_j\colon j=1,\dots,m}$.

Let $\bar\theta$ be a value of the canonical parameter such that 
  \begin{equation*}
    q(x) = 
    \begin{cases}
      \frac{\expof{\sum_{j=1}^m \bar\theta_j T_j(x)}}{\sum_{x\in S} \expof{\sum_{j=1}^m \bar\theta_j T_j(x)} \mu(x)} & \text{if $x\in S$}, \\ 0 & \text{if $x\in\sspace X\setminus S$.}
    \end{cases}
  \end{equation*}
For $\beta \in \reals$,
\begin{equation*}
  \beta A + \sum_{j=1}^m \bar\theta_j T_j = \sum_{j=1}^m (\beta\alpha_j+\bar\theta_j) T_j + \beta \alpha_0,
\end{equation*}
so that the one-dimensional statistical model 
\begin{equation*}
  p_{\beta} = \expof{\beta (A-\alpha_0) + \sum_{j=1}^m \bar\theta_j T_j - \psi(\beta\alpha+\bar\theta)}, \quad \beta\in\reals, 
\end{equation*}
is a sub-model of \eqref{eq:1}. The family of densities
\begin{equation*}
  \frac{p_\beta}{p_0} = \expof{\beta(A-\alpha_0)-\left(\psi(\beta\alpha+\bar\theta)-\psi(\bar\theta)\right)}
\end{equation*}
is a one-dimensional exponential family whose canonical statistics $A-\alpha_0$ reaches its minimum value $-\alpha_0$ on $S$. Therefore, if $\beta_n\to-\infty$, $n\to\infty$, its limit is the uniform distribution on $S$ and, consequently,  $p_{\beta_n}$ is convergent to $q$. \qed
\end{proof}

\section{Extended families}\label{sec:hilbert}
In this section we assume the exponential family \eqref{eq:1} to be of lattice type, {i.e.} we assume that the $m\times n$ matrix $A = [T_j(x)\mu(x)]$, $j=1,\dots,m$ and  $x \in \sspace X$, is non-negative integer valued. Hence, the exponential family can be written as in Equation \eqref{eq:1-T} and takes the monomial parametric form
 \begin{equation}
  \label{eq:1-TT}
  p(x;\zeta) \propto \prod_{j \colon A_j(x) > 0} \zeta_j^{A_j(x)}, \quad \zeta_j\ge0, \quad j=1,\dots,m.
\end{equation}
In \cite{geiger|meek|sturmfels:2006} the statistical model \eqref{eq:1-TT} is called the \emph{$A$-model}, see also \cite{drton|sturmfels|sullivan:2009}. If all $\zeta_j$'s are positive, then \eqref{eq:1-TT} is the exponential family with a different parameterization. If we let one, or more, of the $\zeta_j$'s to be zero, either the monomials in \eqref{eq:1-TT} are zero for all $x \in \sspace X$, in which case no probability is defined, or the monomials are non-zero for some $x$, giving rise to a statistical model with restricted support, see the discussion in \cite{geiger|meek|sturmfels:2006}. 

Each integer vector $k$ such that $Ak=0$, {i.e.} $k \in \ker_{\integers} A$, splits into its positive and negative part, $k = k^+ - k^-$, and we have
\begin{equation}\label{eq:bin-T}
  \prod_{x\colon k^+(x) > 0} p(x,\zeta)^{k^+(x)} = \prod_{x\colon k^-(x) > 0} p(x;\zeta)^{k^-(x)}, \quad k \in \ker_\integers A.
\end{equation}
The statistical model defined by the infinite system of binomial equations \eqref{eq:bin-T} is called the \emph{toric model} of $A$, as defined in \cite{pistone|riccomagno|wynn:2001}. Again, if all the probabilities in \eqref{eq:bin-T} are positive, then the toric model is just the exponential family. If some probabilities are zero, then the toric model implies the $A$-model. In fact, substitution of \eqref{eq:1-TT} into \eqref{eq:bin-T} leads to an algebraic identity, without any restriction on the parameters $\zeta_j$. 

The existence of a finite generating set for Equation \eqref{eq:bin-T} is discussed in details in \cite{geiger|meek|sturmfels:2006}, see also \cite{drton|sturmfels|sullivan:2009}. Moreover, in \cite{geiger|meek|sturmfels:2006} it is proved that each probability in the extended exponential family satisfies \eqref{eq:bin-T}. We shall obtain a related result in a different way.

Consider a second $l \times n$ matrix $B$ with the same integer $\ker$ as $A$. The exponential model would be the same, but the border cases of the $A$-model could be different then the border cases of the $B$-model. The problem of finding a suitable maximal monomial model was considered first in \cite{rapallo:2007toric} and it is fully discussed in \cite{rauh|kahle|ay:09}. Rapallo's method has been applied in \cite{consonni|pistone:2007} to the Bayesian analysis of tables with structural zeros. Here, we show that all of the extended exponential family is actually parameterized by this maximal monomial model. For a related approach see also \cite{rinaldo|feinberg|Zhou:2009}. 

The maximality of the monomial model is defined as follows. Consider the model matrix $A \in \integers^{m\times n}$. If the constant vector $1$ does not belong to the row space switch to the matrix $[{\mathbf 1} A]$ which defines the same exponential model.  Let the column span of the orthogonal matrix $K = [w_1 \cdots w_l] \in \integers^{n\times l}$ be $\ker_\rationals A$. The integer matrix $K$ can be computed by a symbolic algebra software, such as \cite{CocoaSystem,4ti2}. Numeric software might be unsuitable because it will normally produce floating point unit vectors, not integer vectors. 

Consider all possible rows of a non-negative matrix equivalent to $A$, i.e. producing the same statistical model when all the parameters are strictly positive:
\begin{equation*}
  \mathcal B = \setof{b \in \Span_\rationals(A)}{b \ne 0, b \in \integers^n_+} = \setof{b \in \integers^n_+}{b \ne 0, b^TK = 0}.
\end{equation*}
The set $\mathcal B$ is closed for the sum of vectors. It is proved in \cite{schrijver:1986} that a unique inclusion-minimal generating set, called \emph{Hilbert basis}, exists. The Hilbert basis can be computed by symbolic software \cite{CocoaSystem,4ti2}. It is a $\rationals$-generating set but it is usually much larger than a lattice basis.

The following theorem was stated first il \cite{rapallo:2007toric} without a complete proof, see also the discussion in \cite{rauh|kahle|ay:09}.
\begin{theorem}\label{th:hilbert}
Let us consider the set $\mathcal B$ of non-negative and non-zero integer vectors that are orthogonal to $\ker_{\integers} A$ and let $b_1,\dots,b_l$ be its unique Hilbert basis. Define a $l\times n$ matrix $B$ whose rows are the elements of the Hilbert basis. Hence, the extended exponential family is fully parametrized by the $B$-model with non-negative parameters {i.e.} each one of the maximal exposed subsets of the $A$-model is obtained by letting one of the $\zeta_j$'s to be zero.
\end{theorem}
\begin{proof} The constant vector belongs to $\mathcal B$, therefore $1, b_1, \dots, b_l$ is a $\rationals$-vector generating set, possibly non-minimal. In fact, any rational basis of $\Span_\rationals(A)$ becames an integer basis by multiplication by a suitable integer; the integer basis is transformed to an integer positive basis by adding, where needed, a constant integer vector; each of the vector obtained in such a way belong to $\mathcal B$. 

The sets $S_j = \setof{x\in\sspace X}{b_j(x)=0}$, $j=1,\dots,l$ are non empty. In fact, if $m_j = \min_x b_j(x) > 0$, as $b_j(x) \ne 0$ for some $x$, the vector $b_j - m_j\mathbf 1$ belongs to $\mathcal B$, and therefore can be represented as
  \begin{equation*}
    b_j(x) - m_j = \sum_{i=1}^l n_i b_i(x), \quad x\in\sspace X.
  \end{equation*}
If $n_j = 0$, the basis is not minimal. If $n_j \ge 1$, subtracting $b_j(x)$ from both sides, we get, by inspection of the signs of the two sides, that $m_j=0$.

Each of the $S_j$'s is an exposed set of the exponential family. In fact, each element of the Hilbert basis belongs to the row $\rationals-\Span$ of the original matrix $A$, so that
\begin{equation*}
  b_j(x) = \beta_{0j} + \sum_{i=1}^m \beta_{ij} a_i(x), \quad j=1,\dots,l,
\end{equation*}
where $a_i$ is the $i$-th row of $A$. The definition of exposed set is easily checked.

Vice-versa, let $\mathcal S$ be an exposed set, {i.e.}
\begin{equation*}
  b(x) = \beta_0 + \sum_{i=1}^m \beta_{i} a_i(x),
\end{equation*}
with $\mathcal S = \set{x\colon b(x)=0}$ and $b(x) > 0$ for each $x \notin \mathcal S$. As $A$ has integer entries, the coefficients $\beta_0, \beta_1, \dots,\beta_l$ can be chosen to have integer values, therefore $b \in \mathcal B$ and it is a sum of elements of the Hilbert basis, 
\begin{equation*}
  b(x) = \sum_{j=1}^l \alpha_j b_j(x), \quad \alpha_j \in \integers_+,\quad j=1,\dots,l.
\end{equation*}
Therefore, $S = \cap_{j \colon \alpha_j \ne 0} S_j$.\qed
\end{proof}

\begin{remark}
The additive representation of $b$ for maximal exposed sets contains only one term. However, the Hilbert basis might contain an element $b_j$ such that its zero set $S_j$ is the intersection of other $S_j$'s. In such a case, such a $b_j$ could be dropped from the $B$-model without loosing any part of the extended family.
\end{remark}
\section{Examples}

\subsection{4-cycle}
\label{sec:4-cycle}
The 4-cycle is the exponential family
\begin{equation*}
  \expof{\theta_D D + \theta_C C + \theta_B B + \theta_A A + \theta_{BA} BA + \theta_{CB} CB + \theta_{DC} DC + \theta_{AD} AD - \psi(\theta)}
\end{equation*}
where $A,B,C,D$ are the coordinates of $\mathcal X = \set{\pm 1}^4$. The model matrix $A$ and the $B$ matrix are shown in the following edited R output:
\begin{equation*}\tiny
  \begin{array}{c|rrrrrrrrr|cccccccccccccccccccccccc}
  \mathcal X &
I & D & C & B & A & BA & CB & DC & DA & b_{1} & b_{2} & b_{3} & b_{4} & b_{5} & b_{6} & b_{7} & b_{8} & b_{9} & b_{10} & b_{11} & b_{12} & b_{13} & b_{14} & b_{15} & b_{16} & b_{17} & b_{18} & b_{19} & b_{20} & b_{21} & b_{22} & b_{23} & b_{24}\\
++++ & 1 & 1 & 1 & 1 & 1 & 1 & 1 & 1 & 1 & 0 & 1 & 1 & 1 & 0 & 1 & 0 & 0 & 0 & 0 & 0 & 0 & 0 & 0 & 0 & 1 & 1 & 0 & 0 & 1 & 0 & 0 & 1 & 0\\
+++- & 1 & 1 & 1 & 1 & -1 & -1 & 1 & 1 & -1 & 0 & 1 & 1 & 0 & 0 & 0 & 0 & 0 & 0 & 1 & 1 & 1 & 0 & 0 & 0 & 0 & 1 & 0 & 0 & 1 & 0 & 1 & 0 & 0\\
++-+ & 1 & 1 & 1 & -1 & 1 & -1 & -1 & 1 & 1 & 0 & 1 & 0 & 1 & 0 & 0 & 0 & 1 & 0 & 1 & 0 & 0 & 0 & 0 & 0 & 0 & 0 & 0 & 1 & 1 & 0 & 0 & 1 & 1\\
++-- & 1 & 1 & 1 & -1 & -1 & 1 & -1 & 1 & -1 & 0 & 1 & 0 & 0 & 0 & 0 & 0 & 0 & 0 & 1 & 1 & 0 & 0 & 1 & 0 & 0 & 1 & 1 & 1 & 0 & 0 & 0 & 1 & 0\\
+-++ & 1 & 1 & -1 & 1 & 1 & 1 & -1 & -1 & 1 & 1 & 0 & 0 & 1 & 0 & 1 & 1 & 1 & 0 & 0 & 0 & 0 & 0 & 1 & 0 & 1 & 0 & 0 & 0 & 0 & 0 & 0 & 1 & 0\\
+-+- & 1 & 1 & -1 & 1 & -1 & -1 & -1 & -1 & -1 & 1 & 0 & 0 & 0 & 0 & 0 & 1 & 1 & 0 & 1 & 1 & 1 & 0 & 1 & 0 & 0 & 0 & 0 & 0 & 0 & 0 & 1 & 0 & 0\\
+--+ & 1 & 1 & -1 & -1 & 1 & -1 & 1 & -1 & 1 & 1 & 0 & 0 & 1 & 0 & 0 & 0 & 1 & 0 & 0 & 0 & 0 & 1 & 0 & 0 & 1 & 0 & 0 & 0 & 1 & 0 & 1 & 0 & 1\\
+--- & 1 & 1 & -1 & -1 & -1 & 1 & 1 & -1 & -1 & 1 & 0 & 0 & 0 & 0 & 0 & 0 & 0 & 0 & 0 & 1 & 0 & 1 & 1 & 0 & 1 & 1 & 1 & 0 & 0 & 0 & 1 & 0 & 0\\
-+++ & 1 & -1 & 1 & 1 & 1 & 1 & 1 & -1 & -1 & 0 & 0 & 1 & 0 & 1 & 1 & 0 & 0 & 0 & 0 & 0 & 0 & 0 & 1 & 0 & 1 & 1 & 0 & 0 & 0 & 1 & 1 & 0 & 0\\
-++- & 1 & -1 & 1 & 1 & -1 & -1 & 1 & -1 & 1 & 0 & 0 & 1 & 0 & 0 & 0 & 0 & 1 & 1 & 0 & 0 & 1 & 0 & 0 & 0 & 1 & 0 & 0 & 0 & 1 & 1 & 1 & 0 & 0\\
-+-+ & 1 & -1 & 1 & -1 & 1 & -1 & -1 & -1 & -1 & 0 & 0 & 0 & 0 & 1 & 0 & 0 & 1 & 0 & 1 & 0 & 0 & 0 & 1 & 0 & 0 & 0 & 0 & 1 & 0 & 1 & 1 & 0 & 1\\
-+-- & 1 & -1 & 1 & -1 & -1 & 1 & -1 & -1 & 1 & 0 & 0 & 0 & 0 & 0 & 0 & 0 & 1 & 1 & 0 & 0 & 0 & 0 & 1 & 0 & 1 & 0 & 1 & 1 & 0 & 1 & 0 & 1 & 0\\
--++ & 1 & -1 & -1 & 1 & 1 & 1 & -1 & 1 & -1 & 0 & 0 & 0 & 0 & 1 & 1 & 1 & 0 & 0 & 1 & 0 & 0 & 0 & 1 & 1 & 0 & 1 & 0 & 0 & 0 & 0 & 0 & 1 & 0\\
--+- & 1 & -1 & -1 & 1 & -1 & -1 & -1 & 1 & 1 & 0 & 0 & 0 & 0 & 0 & 0 & 1 & 1 & 1 & 1 & 0 & 1 & 0 & 0 & 1 & 0 & 0 & 0 & 0 & 1 & 0 & 0 & 1 & 0\\
---+ & 1 & -1 & -1 & -1 & 1 & -1 & 1 & 1 & -1 & 0 & 0 & 0 & 0 & 1 & 0 & 0 & 0 & 0 & 1 & 0 & 0 & 1 & 0 & 1 & 0 & 1 & 0 & 0 & 1 & 0 & 1 & 0 & 1\\
---- & 1 & -1 & -1 & -1 & -1 & 1 & 1 & 1 & 1 & 0 & 0 & 0 & 0 & 0 & 0 & 0 & 0 & 1 & 0 & 0 & 0 & 1 & 0 & 1 & 1 & 1 & 1 & 0 & 1 & 0 & 0 & 1 & 0\\
\end{array}
\end{equation*}
In this exemple all the $b_j$ vectors are binary vectors, which implies they are all indispensable. The vectors $F_j=1-b_j$ are the indicator functions of the $S_j$ sets. The polynomial representation is (after multiplication by $1/16$):
\begin{equation*}\tiny
  \begin{array}{c|rrrrrrrrrrrrrrrrrrrrrrrr}
 \theta &
F_{1} & F_{2} & F_{3} & F_{4} & F_{5} & F_{6} & F_{7} & F_{8} & F_{9} & F_{10} & F_{11} & F_{12} & F_{13} & F_{14} & F_{15} & F_{16} & F_{17} & F_{18} & F_{19} & F_{20} & F_{21} & F_{22} & F_{23} & F_{24}\\
I & 12 & 12 & 12 & 12 & 12 & 12 & 12 & 8 & 12 & 8 & 12 & 12 & 12 & 8 & 12 & 8 & 8 & 12 & 12 & 8 & 12 & 8 & 8 & 12\\
D & -4 & -4 & 0 & -4 & 4 & 0 & 0 & 0 & 4 & 0 & -4 & 0 & 0 & 0 & 4 & 0 & 0 & 0 & 0 & 0 & 4 & 0 & 0 & 0\\
C & 4 & -4 & -4 & 0 & 0 & 0 & 4 & 0 & 0 & 0 & 0 & 0 & 4 & 0 & 4 & 0 & 0 & 0 & -4 & 0 & -4 & 0 & 0 & 0\\
B & 0 & 0 & -4 & 0 & 0 & -4 & -4 & 0 & 0 & 0 & 0 & -4 & 4 & 0 & 0 & 0 & 0 & 4 & 4 & 0 & 0 & 0 & 0 & 4\\
A & 0 & 0 & 0 & -4 & -4 & -4 & 0 & 0 & 4 & 0 & 4 & 4 & 0 & 0 & 0 & 0 & 0 & 4 & 0 & 0 & 0 & 0 & 0 & -4\\
BA & 0 & 0 & 0 & 0 & 0 & -4 & 0 & 4 & 0 & 4 & 0 & 4 & 0 & -4 & 0 & -4 & -4 & -4 & 0 & 4 & 0 & 4 & -4 & 4\\
CB & 0 & 0 & -4 & 0 & 0 & 0 & 4 & 4 & 0 & 4 & 0 & 0 & -4 & 4 & 0 & -4 & -4 & 0 & 4 & -4 & 0 & -4 & 4 & 0\\
DC & 4 & -4 & 0 & 0 & 0 & 0 & 0 & 4 & 0 & -4 & 0 & 0 & 0 & 4 & -4 & 4 & -4 & 0 & 0 & -4 & 4 & 4 & -4 & 0\\
DA & 0 & 0 & 0 & -4 & 4 & 0 & 0 & -4 & -4 & 4 & 4 & 0 & 0 & 4 & 0 & -4 & 4 & 0 & 0 & -4 & 0 & 4 & -4 & 0\\
 \end{array}
\end{equation*}
e.g.
\begin{equation*}
  F_1 = \frac34 - \frac14 D + \frac14 C + \frac14 DC
\end{equation*}
that is $DC = D-C$ on $S_1$.

The Gr\"obner basis of each ideal $\langle A^2-1, B^2-1, C^2-1, D^2-1,F_j-1\rangle$ reveals in a different way the aliasing induced on each facet.

Next: polynomial representation of the model.

\subsection{Markov chain}
\label{sec:markov-chain}
Let $X_t$, $t=0,1,\dots,n$ be a Markov chain with stationary transitions on the binary state space $\set{0,1}$. Let us denote by $t_x=\probof{X_0=x}$, $x=0,1$, the
initial probability and with $t_{xy} = \probof{X_1 = y | X_0 = x}$, $x,y=0,1$, the transition probabilities. For each trajectory $\omega \in \sspace X = \set{0,1}^{n+1}$ the probability of the trajectory is
\begin{equation}\label{eq:2}
  p(\omega) = t_0^{(1-X_0(\omega))}t_1^{X_0} \prod_{x,y=0}^1 t_{xy}^{N_{xy}(\omega)},
\end{equation}
where $N_{xy}(\omega)$ is the number of transitions from $x$ to $y$ appearing in the trajectory $\omega$.

This in an instance of the toric model of the $\sspace X \times 6$ matrix whose rows are
\begin{equation*}
  \begin{bmatrix}
    (1-X_0) & X_0 & N_{00} & N_{01} & N_{10} & N_{11}
  \end{bmatrix}.
\end{equation*}

Let us compute the confounding, i.e. find the vectors $c \in \reals^6$ such that
\begin{multline*}
  c_0 (1-X_0(\omega)) + c_1 X_0(\omega) \\ + c_{00} N_{00}(\omega) + c_{01}N_{01}(\omega) + C_{10} N_{10}(\omega) + c_{11} N_{11}(\omega) = \alpha, \quad \omega \in \sspace X.
\end{multline*}
for some $\alpha$. Note the following equalities:
\begin{align*}
N_{00} &= \sum_{t=1}^n (1-X_{t-1})(1-X_{t}) = n - X_0 - 2\sum_{t=1}^{n-1} X_{t} - X_n + \sum_{t=1}^n X_{t-1}X_{t}, \\
N_{01} &= \sum_{t=1}^n (1-X_{t-1}) X_{t} = \sum_{t=1}^{n-1} X_{t} + X_n - \sum_{t=1}^n X_{t-1}X_{t}, \\
N_{10} &= \sum_{t=1}^n X_{t-1}(1-X_{t}) = X_0 + \sum_{t=1}^{n-1} X_{t} - \sum_{t=1}^n X_{t-1}X_{t}, \\
N_{11} &= \sum_{t=1}^n X_{t-1}X_{t}. \\
\end{align*}
Expanding the equation for $c$ and observing that the vectors $1$, $X_0$, $\sum_{t=1}^{n-1} X_{t}$, $X_n$, $\sum_{t=1}^nX_{t-1}X_{t}$ are linearly independent, we obtain, equating to zero the coefficient of each vector, that
\begin{gather*}
  c_0 + nc_{00} = \alpha \\
  c_1 - c_0 - c_{00} + c_{10} = 0 \\
  -2c_{00} + c_{01} + c_{10} = 0 \\
  -c_{00} + c_{01} = 0 \\
  c_{00} - c_{01} - c_{10} + c_{11} = 0
\end{gather*}
The solution of the previous system is
\begin{equation*}
  c_0 = c_1, \quad c_{00} = c_{01} = c_{10} = c_{11}.
\end{equation*}

It follows that an identifiable parameterization of the exponential model from \eqref{eq:2} is
\begin{equation}
  \label{eq:3}
    p(\omega) = t_0^{(1-X_0(\omega))}t_1^{X_0} \prod_{x,y=0}^1 t_{xy}^{N_{xy}(\omega)}, \quad t_0+t_1 = 1, \quad \sum_{xy} t_{xy} = 2,
\end{equation}
while the Markov chain model is the submodel
\begin{equation}
  \label{eq:4}
    p(\omega) = t_0^{(1-X_0(\omega))}t_1^{X_0} \prod_{x,y=0}^1 t_{xy}^{N_{xy}(\omega)}, \quad t_0+t_1 = 1, \quad t_{00}+t_{01} = 1, \quad t_{10}+t_{11} = 1.
\end{equation}

The orthogonal space of the model matrix is generated by the vector $k = (n,n,1,1,1,1)$
%

\begin{acknowledgement}
Part of this work was prepared by L. Malag\`o during a visit to the group led by Nihat Ay at Max Planck Institute for Mathematics in the Sciences. L. Malag\`o thanks his PhD supervisor M. Matteucci for support. G. Pistone acknowledges the support of DIMAT Politecnico di Torino, SAMSI NC, and Collegio Carlo Alberto Moncalieri. The authors thank T. Kahle for suggestions and comments on an early draft and S. Sullivan for critical comments about Theorem \ref{th:hilbert}. A shorter version of this paper was presented at SIS2010 Padova.
\end{acknowledgement}

\bibliographystyle{spmpsci}

\end{document}